\documentclass[11pt]{amsart}
\usepackage{graphicx}
\usepackage[active]{srcltx}
 \makeatletter
\renewcommand*\subjclass[2][2000]{%
  \def\@subjclass{#2}%
  \@ifundefined{subjclassname@#1}{%
    \ClassWarning{\@classname}{Unknown edition (#1) of Mathematics
      Subject Classification; using '1991'.}%
  }{%
    \@xp\let\@xp\subjclassname\csname subjclassname@#1\endcsname
  }%
}
 \makeatother
\usepackage{enumerate,url,amssymb,  mathrsfs}

\newtheorem{theorem}{Theorem}[section]
\newtheorem{lemma}[theorem]{Lemma}
\newtheorem*{lemma*}{Lemma}

\theoremstyle{definition}

\newtheorem{conjecture}[theorem]{Conjecture}

\theoremstyle{remark}
\newtheorem{remark}[theorem]{Remark}

\numberwithin{equation}{section}

\def\XXint#1#2#3{{\setbox0=\hbox{$#1{#2#3}{\int}$}
\vcenter{\hbox{$#2#3$}}\kern-.5\wd0}}

\begin{document}

\title{{The Norm of the Harmonic Bergman projection and Besov spaces }}



 \author{Djordjije Vujadinovi\'c}
\address{ University of Montenegro, Faculty of Mathematics, Dzordza Va\v singtona  bb, 81000 Podgorica, Montenegro}
 \email{  djordjijevuj@t-com.me}
 \date{}
\begin{abstract}
We estimate the norm of the harmonic Bergman projection in the context of harmonic Besov spaces. We  obtain the two-side norm estimates in general $L^{p}-$case.
 \end{abstract}

\maketitle

\section{Introduction and Notation}

By $\mathbf{B}$ we denote the open unit ball in $\mathbf{R}^{n}$ for a fixed positive integer $n\geq 2.$ For $\alpha >0,$ the weighted measure $dv_{\alpha}$ on  $\mathbf{B}$  is defined by $dv_{\alpha}(x)=c_{\alpha}(1-|x|^{2})^{\alpha-1}dv(x)$ where $dv(x)$ is the Lebesgue volume measure on $\mathbf{B}.$ The constant $c_{\alpha}$ is chosen such that  $v_{\alpha}$ has a total mass 1, and it is given by
$$c_{\alpha}=\frac{2\Gamma(\frac{n}{2}+\alpha)}{n|\mathbf{B}|\Gamma(\frac{n}{2})\Gamma(\alpha)}.$$
Here $|\mathbf{B}|$ denotes the volume of $\mathbf{B}.$
Also, $d\tau(x)=(1-|x|^{2})^{-n}dv(x).$

The space of all harmonic functions $f$ on the unit ball $\mathbf{B}$  is denoted by $H(\mathbf{B})$ and we use the following notation for the partial derivatives $$|\partial^{m}f(x)|=\sum_{|k|=m}|\partial^{\alpha} f(x)|,\enspace \partial^{k}f(x)=\frac{\partial^{m}f(x)}{\partial x^{k}}.$$
Here  $k=(k_1,...,k_n)$ is multi-index and $|k|=\sum_{i=1}^{n}k_i=m.$

The spaces under consideration  through this paper are so-called the Besov spaces of harmonic functions in the unit ball.  Harmonic Besov spaces have been
studied early from different point of views in general settings of conditions and domains. For instance one can see the extensive study in \cite{kaptan}.

In this paper we follow the work of M.Jevti\'c and M.Pavlovi\'c (see \cite{Pavlov1}) on  the harmonic Besov space $B^{p}(\mathbf{B})$ $1\leq p\leq \infty$ which is defined to be the space consisting  of all harmonic functions $f\in  H(\mathbf{B})$ such that the function
$(1-|x|^{2})^{k}|\partial^{k} f(x)|$ belongs to $L^{p}(\mathbf{B},d\tau)$ for some positive integer $k,$ $k>\frac{n-1}{p}.$

More precisely, the next theorem holds:
\begin{theorem}
\label{important}
Let $1\leq p\leq \infty.$  If $f\in H(\mathbf{B}),$ then the following statements are equivalent:\\
a)There exists a positive integer $m_{0}>\frac{n-1}{p}$ such that $(1-|x|^{2})^{m_0}|\partial^{m_0}f(x)|\in L^{p}(\mathbf{B},d\tau).$\\
b)For all positive integers $m>\frac{n-1}{p}, (1-|x|^{2})^{m}|\partial^{m}f(x)|\in L^{p}(\mathbf{B},d\tau).$
\end{theorem}
 The definition is independent from choice of integer $k.$

 We should say that the harmonic Besov space and the norm can be characterize in terms of certain fractional differential operators (see \cite{Pavlov1},Theorem 3.1). However, we restrict our attention to the norm provided by the classical partial derivatives.

 The Besov space-norm is defined  in following manner:
\begin{equation}
\label{bitnoo}
\|f\|_{B^{p}}=\sum_{|\alpha|<m}|\partial^{\alpha}f(0)|+\left(\int_{\mathbf{B}}(1-|x|^{2})^{mp}|\partial^{m}f(x)|^{p}d\tau(x)\right)^{1/p},\enspace 1\leq p<\infty,
\end{equation}
\begin{equation}
\label{bitnoo1}
\|f\|_{B^{\infty}}=\sum_{|\alpha|<m}|\partial^{\alpha}f(0)|+\sup_{x\in \mathbf{B}}(1-|x|^{2})^{m}|\partial^{m}f(x)|,\enspace p=\infty.
\end{equation}
We will denote by the same denotement of the norm \eqref{bitnoo1} the derived semi-norm:
\begin{equation}\label{normaprava1}
\|f\|_{B^{p}}=\left(\int_{\mathbf{B}}(1-|x|^{2})^{mp}|\partial^{m}f(x)|^{p}d\tau(x)\right)^{1/p}.
\end{equation}

In the rest of the paper we will consider only the semi-norm \eqref{normaprava1} and appropriate operator norm:
\begin{equation*}
\|T\|_{L^{p}(\mathbf{B},d\tau)\rightarrow B^{p}}=\sup_{\|f\|_{L^{p}(\mathbf{B},d\tau)}\neq 0}\frac{\|Tf\|_{B^{p}}}{\|f\|_{L^{p}(\mathbf{B},d\tau)}},
\end{equation*}
for a linear mapping $T: L^{p}(\mathbf{B},d\tau)\rightarrow B^{p}.$

{\bf Weighted harmonic Bergman kernel.}
The weighted harmonic Bergman kernel  arises in the context of weighted harmonic Hilbert-Bergman space $b_{\alpha}^{2}=b_{\alpha}^{2}(\mathbf{B}),$ which is the space of all complex-valued  $f$ functions such that the norm $$\|f\|_{L^{2}(\mathbf{B},dv_{\alpha})}=\left(\int_{\mathbf{B}}|f|^{2}dv_{\alpha}\right)^{1/2}<\infty$$ is finite.

Since the Hilbert space $b_{\alpha}^{2}$ is closed in the weighted Lebesgue space $L^{2}(\mathbf{B},dv_{\alpha})$ there is a unique reproducing kernel $R_{\alpha}(x,y)$ defined on $\mathbf{B}\times \mathbf{B}$ which is real and symmetric, such that $R_{\alpha}(x,\cdot)\in b_{\alpha}^{2}$ for any fixed $x\in \mathbf{B}$ and
$$f(x)=\int_{\mathbf{B}}f(y)R_{\alpha}(x,y)dv_{\alpha}(y),\enspace f\in b_{\alpha}^{2} .$$

Following \cite{Axler} (see [3,Chapter 5]) the unweighted harmonic Bergman kernel is given by the sequent formula
\begin{equation}
\label{netezina}
R_{0}(x,y)=\frac{1}{n|\mathbf{B}|}\sum_{k=0}^{\infty}(n+2k)Z_{k}(x,y),\enspace x,y\in \mathbf{B},
\end{equation}
where $Z_{k}$ are extended zonal harmonic.

The series \eqref{netezina} converges absolutely and uniformly on $K\times \mathbf{B}$ for every compact set $K\subset \mathbf{B}.$

More generally, the formula for the weighted harmonic Bergman kernel for $\alpha>0$ is given by
 \begin{equation}
\label{tezina}
R_{\alpha}(x,y)=\omega_{\alpha}\sum_{k=0}^{\infty}\frac{\Gamma(k+n/2+\alpha)}{\Gamma(k+n/2)}Z_{k}(x,y),\enspace x,y\in \mathbf{B},
\end{equation}
where $\omega_{\alpha}=\frac{\Gamma(n/2)}{\Gamma(n/2+\alpha)}.$ We will denote above coefficients, with $A_{k}^{n,\alpha}=\frac{\Gamma(k+n/2+\alpha)}{\Gamma(k+n/2)}.$

The weighted harmonic Bergman projection $P_{\alpha}$ is realized as an integral operator
$$P_{\alpha}f(x)=\int_{\mathbf{B}}f(y)R_{\alpha}(x,y)dv_{\alpha}(y),\enspace f\in L^{2}(\mathbf{B},dv_{\alpha}), x\in \mathbf{B}\enspace .$$
It is known that $P_{\alpha}$ is bounded on $L^{p}(\mathbf{B},dv_{\alpha})$ for $1<p<\infty.$

In the sequel, we would like to light the connection between the harmonic Besov space and the harmonic Bergman projection $P_{\alpha},$ which is expressed in the next theorem (see \cite{Pavlov1}).

\begin{theorem}
\label{karakter}
For the $1\leq p\leq \infty,$ the Bergman projection $P_{\alpha},\enspace \alpha>0,$ maps $L^{p}(\mathbf{B},d\tau)$ boundedly onto the harmonic Besov space $B^{p}.$
\end{theorem}

 The main goal of this paper is related to  the norm  estimations of the Bergman projection
 $$P_{\alpha}:L^{p}(\mathbf{B},d\tau)\rightarrow B^{p}, \enspace 1<p<\infty.$$

 The analogous estimates in the analytic case of Besov space in one-dimensional case were considered in \cite{djvu}.

\section*{Preliminaries}

{\bf Zonal harmonics}
In the introductory section we mentioned the zonal harmonics $Z_{j}(\xi,\eta),$ $\xi,\eta\in \mathbb{S}$ ($j$ is non-negative integer).
For more information about zonal harmonic we refer to \cite{Axler} (Chapter 5).

Since the space of all harmonic polynomials of degree $j$  $H_{j}(\mathbb{S})$ in $\mathbb{R}^{n}$ is a finite-dimensional inner-product space, the unique function $Z_{j}(\cdot,\eta)\in H_{j}(\mathbb{S})$ guaranteed by the classical Riesz theorem such that
$$p(\eta)=\int_{\mathbb{S}}p(\xi)\overline{Z_{j}(\xi,\eta)}d\sigma(\xi),$$
for all $p\in H_{j}(\mathbb{R}^{n})$ is called zonal harmonic.

$Z_{s}$ is a real valued function, which is symmetric, i.e., $Z_{j}(\xi,\eta)=Z_{s}(\eta,\xi).$

Also, the following properties hold
\begin{equation}
\label{zona1}
Z_{j}(\xi,\xi)=\dim{H_{j}(\mathbb{R}^{n})},
\end{equation}
\begin{equation}
\label{zona1}
|Z_{j}(\xi,\eta)|\leq \dim{H_{j}(\mathbb{R}^{n})}.
\end{equation}
where dimension of space $H_{j}(\mathbb{R}^{n})$ is known and given
\begin{equation}
\label{dimension}
\dim{H_{j}(\mathbb{R}^{n})}={n+j-1\choose n-1}-{n+j-3\choose n-1}.
\end{equation}

Moreover, we can express the explicit formula for the zonal harmonic. Namely, for $j>0,$ $x\in \mathbb{R}^{n}$ and $\xi\in \mathbb{S},$ then
$$Z_{j}(x,\xi)=(n+2j-2)\sum_{i=0}^{[j/2]}(-1)^{i}\frac{n(n+2)\cdot\cdot\cdot(n+2j-2i-4)}{2^{i}i!(j-2i)!}(x\cdot\xi)^{j-2i}|x|^{2i}.$$

{\bf Hypergeometric series}

 The hypergeometric function
${}_2F_{1}(a, b; c; t)$ is defined by the series expansion
$$\sum_{n=0}^{\infty}\frac{(a)_{n}(b)_{n}}{n!(c)_{n}}t^{n},\enspace \mbox{for}\enspace |t|<1,$$
and by the continuation elsewhere. Here $(a)_{n}$ denotes the shifted factorial, i.e.,
$(a)_{n} = a(a + 1)\cdot\cdot\cdot(a + n -1)$ with any real number a.

We recall some known identities for the hypergeometric function (for details se \cite{anderson}).

    Euler's identity:
      \begin{equation}\label{en1}{}F(a,b;c;x)=(1-x^2)^{c-a-b}{}F(c-a,c-b;c;x),\enspace \mbox{Re}{(c)}>\mbox{Re}{(b)}>0,\end{equation}
Gauss's identity:
\begin{equation}\label{gaus1}{}F(a,b;c;1)=\frac{\Gamma(c)\Gamma(c-a-b)}{\Gamma(c-a)\Gamma(c-b)},\enspace \mbox{Re}{(c-a-b)}>0,\end{equation}
   Differentiation identity:
 \begin{equation}\label{izvod}\frac{\partial}{\partial x}{}F(a,b;c;x)=\frac{ab}{c}{}F(a+1,b+1;c+1;x).\end{equation}

\section{The General $L^{p}-$ case}

According to the Theorem \ref{karakter} the harmonic Bergman projection maps the space $L^{p}(\mathbf{B},d\tau)$ onto the Besov space $B^{p}.$

Therefore for any $g\in B^{p},$ there exists $f\in L^{p}(\mathbf{B},d\tau)$ such that
\begin{equation}
\label{onto}
g(x)=P_{\alpha}f(x)=\int_{\mathbf{B}}R_{\alpha}(x,y)f(y)dv_{\alpha}(y),
\end{equation}
and

\begin{equation}
\label{onto1}
\partial_{x}^{m}g(x)=\int_{\mathbf{B}}\partial_{x}^{m} R_{\alpha}(x,y)f(y)dv_{\alpha}(y).
\end{equation}
It is clear that the growth rate of the derivative $\partial_{x}^{m}R_{\alpha}(x,y)$ will play a central role in the estimating the norm for the harmonic Bergman projection $P_{\alpha}.$

Let us mention that the first results in this direction were done in \cite{Pavlov}.
The authors proved the existence of the constants $C_{\alpha},\enspace C_{\alpha}^{m}$  in the following inequalities:

\begin{equation}
\label{prva}
|R_{\alpha}(x,y)|\leq C_{\alpha}|x-y|^{-n+1-\alpha},\enspace  |x|<1,\enspace\mbox{and}\enspace |y|=1.
\end{equation}

\begin{equation}
\label{druga}
|\partial_{x}^{m}R_{\alpha}(x,y)|\leq C_{\alpha}^{m}|rx-\xi|^{-n+1-\alpha-m},\enspace 0\leq r<1,\enspace\mbox{and}\enspace y=r\xi, |\xi|=1.
\end{equation}

The inequality \eqref{druga} give us inducement to observe the integral operators with the kernel $[x,y]^{-n+1-\alpha-m},$ where $$[x,y]=(1-2x\cdot y+|x|^{2}|y|^{2})^{1/2}.$$ More explicitly, we treat the operators
\begin{equation}
\label{pomocni1}
T_{k}^{\alpha}f(x)=\int_{\mathbf{B}}\frac{f(y)}{[x,y]^{n+\alpha+|k|-1}}dv_{\alpha}(y).
\end{equation}

The problem of estimating the norm in the initial case
\begin{equation*}
T^{\alpha}f(x)=\int_{\mathbf{B}}\frac{f(y)}{[x,y]^{n+\alpha-1}}dv_{\alpha}(y),
\end{equation*}
in the context of harmonic Bergman spaces over the unit ball, is done in \cite{Boo1}.

In the Theorem \eqref{thm1} we give the two-side norm estimate for the operator $T_{k}^{\alpha}.$ To do this we first established the Lemma \ref{lablem1} and Lemma \ref{opasno}.

Let us denote by
\begin{equation}
\label{bitno}
I_{\alpha,s}(x)=\int_{\mathbf{B}}\frac{(1-|y|^{2})^{\alpha-1}}{[x,y]^{n+\alpha+s-1}}dv(y),\enspace |x|<1,
\end{equation}
where, $\alpha> 0$ and $s<0$ except integers less than or equal to zero.

In \cite{Boo} authors considered the asymptotic behaviour for the function $I_{\alpha,s}$  when $|x|\rightarrow 1^{-}$ in more general setting of parameters $\alpha$ and $s.$

More precisely, they showed for any real number $s$ and $\alpha>-1$ that the following asymptotic relations hold
\begin{equation}
\label{studia}I_{\alpha,s}(x)\approx\left\{\begin{array}{rrr}
(1-|x|^{2})^{s},& s>0,\\
1-\log{(1-|x|^{2})},& s=0,\\
c,& s<0.
\end{array}
\right.\end{equation}
Here the the relation $X\approx Y$ means that the coefficient $X/Y$ is bounded by some positive constants from above and belove.

 Also, the sharp estimations of the function $I_{\alpha,s}$  were done in \cite{Boo1} in case when $0<\alpha<\nu,$ $0<s<\nu$ for the given $\nu.$

 The use of hypergeometric function in the mentioned  research is determined by the identity
 \begin{equation}
 \label{liu}
 \int_{\mathbb{S}}\frac{d\sigma(\xi)}{|x-\xi|^{c}}={}_2F_{1}\left(\frac{c}{2},\frac{c-n}{2}+1,\frac{n}{2};|x|^{2}\right),x\in \mathbf{B}
 \end{equation}
 for $c$ real (see \cite{liuu}, Lemma 2.1).

Since from \eqref{studia} we see that $$I_{\alpha,s}(x)=O(1),\enspace \mbox{as}\enspace |x|\rightarrow 1^{-},$$  for $s<0,$ we are interested in finding the upper and lower  bound for the function $I_{\alpha,s}(x).$

The following Lemma gives the answer.
\begin{lemma}
\label{lablem1}
For the function $I_{\alpha,s}(x)$ defined in \eqref{bitno}, where $-1<s+\alpha,$  the following identities hold
$$\max_{0\leq x\leq 1}I_{\alpha,s}(x)= C(\alpha,s),$$

\begin{equation}
\label{konstanta}
C(\alpha,s)=\frac{\pi^{n/2}\Gamma(\alpha)\Gamma(-s)}{\Gamma(\frac{\alpha-s+1}{2})\Gamma(\frac{n+\alpha-s-1}{2})},
\end{equation}
and
\begin{equation}
\label{konstanta1}
\min_{0\leq x\leq 1}I_{\alpha,s}(x)=\frac{\pi^{n/2}\Gamma(\alpha)}{\Gamma(\frac{n}{2}+\alpha)}.
\end{equation}
\end{lemma}

\begin{proof}
The identity \eqref{liu} and the  uniform expansion of the hypergeometric function give
\begin{equation*}
\begin{split}
&I_{\alpha,s}(x)\\
&=n|\mathbf{B}|\int_{0}^{1}(1-t^{2})^{\alpha-1}t^{n-1}{}_2F_{1}\left(\frac{n+\alpha+s-1}{2},\frac{\alpha+s-1}{2}+1,\frac{n}{2};|x|^{2}t^{2}\right)dt\\
&=n|\mathbf{B}|\sum_{i=0}^{\infty}\frac{a_{i}}{i!}|x|^{2i}\\
\end{split}
\end{equation*}
where
\begin{equation*}
\begin{split}
a_{i}&=\frac{(\frac{n+\alpha+s-1}{2})_{i}(\frac{s+\alpha+1}{2})_{i}}{2(\frac{n}{2})_{i}}\int_{0}^{1}(1-t)^{\alpha-1}t^{\frac{n}{2}-1+i}dt\\
&=\frac{(\frac{n+\alpha+s-1}{2})_{i}(\frac{s+\alpha+1}{2})_{i}\Gamma(\alpha)\Gamma(\frac{n}{2}+i)}{2(\frac{n}{2})_{i}\Gamma(\frac{n}{2}+\alpha+i)}.
\end{split}
\end{equation*}
Thus,
$$I_{\alpha,s}(x)=\frac{n|\mathbf{B}|\Gamma(\alpha)\Gamma(\frac{n}{2})}{2\Gamma(\frac{n}{2}+\alpha)}{}_2F_{1}\left(\frac{n+\alpha+s-1}{2},\frac{\alpha+s+1}{2},\alpha+\frac{n}{2},|x|^{2}\right).$$
Further, the fact that the hypergeometric function ${}_2F_{1}(a,b,c;x), a>0,b>0,c>0, x\in [0,1]$ is monotone increasing with respect to $x$ (the property \eqref{izvod}) gives:
\begin{equation*}
\begin{split}
\max_{0\leq x\leq 1}I_{\alpha,s}(x)&=\frac{n|\mathbf{B}|\Gamma(\alpha)\Gamma(\frac{n}{2})}{2\Gamma(\frac{n}{2}+\alpha+1)}{}_2F_{1}\left(\frac{n+\alpha+s-1}{2},\frac{\alpha+s+1}{2},\alpha+\frac{n}{2};1\right)\\
&=\frac{\pi^{n/2}\Gamma(\alpha)\Gamma(-s)}{\Gamma(\frac{\alpha-s+1}{2})\Gamma(\frac{n+\alpha-s-1}{2})}.
\end{split}
\end{equation*}
The same argument implies:
\begin{equation*}
\begin{split}
\min_{0\leq x\leq 1}I_{\alpha,s}(x)&=\frac{n|\mathbf{B}|\Gamma(\alpha)\Gamma(\frac{n}{2})}{2\Gamma(\frac{n}{2}+\alpha)}{}_2F_{1}\left(\frac{n+\alpha+s-1}{2},\frac{\alpha+s+1}{2},\alpha+\frac{n}{2},0\right)\\
&=\frac{\pi^{n/2}\Gamma(\alpha)}{\Gamma(\frac{n}{2}+\alpha)}.
\end{split}
\end{equation*}
\end{proof}

According to the introduced norm \eqref{normaprava1}
we are going to treat the acting of the operator $T_{k}^{\alpha}$ in the following Lebesgue-space settings $$T_{k}^{\alpha}:L^{p}(\mathbf{B},d\tau)\rightarrow L^{p}(\mathbf{B},dv_{mp-n}),\enspace p>1,$$

where $dv_{mp}(x)=(1-|x|^{2})^{mp-n}dv(x),$ $|k|=m>\frac{n-1}{p}.$

We now turn to the finding the upper estimate for the norm of the operator $T_{^k}^{\alpha}.$

\begin{lemma}
\label{opasno}
For the defined operator $T_{k}^{\alpha}$  in \eqref{pomocni1} there is a constant $D_{p}^{|k|}$ such that
$$\|T_{k}^{\alpha}\|_{L^{p}(\mathbf{B},d\tau)\rightarrow L^{p}(\mathbf{B},d\tau)}\leq D_{p}^{|k|},$$ where
\begin{equation}
\label{tojeto1}
D_{p}^{|k|}=\frac{\Gamma(\frac{n}{2}+\alpha)\Gamma^{1/p}(\frac{p}{q}(n+|k|+\alpha-1))\Gamma^{1/q}(\frac{q}{p}|k|)}{\Gamma(\frac{n+\alpha-1}{q}+|k|+\frac{|k|+\alpha+1}{2})\Gamma(\frac{n+\alpha-1}{q}+|k|+\frac{|k|+n+\alpha-1}{2})},
\end{equation}
and $$\frac{1}{p}+\frac{1}{q}=1.$$
\end{lemma}

\begin{proof}
Let us denote by $L^{p}(r\mathbf{B},d\tau),0<r<1,$ the space defined as
$$L^{p}(r\mathbf{B},d\tau)=\{\psi(x)\chi_{r\mathbf{B}}(x)|\psi \in L^{p}(\mathbf{B}, d\tau)\},$$
where $\chi_{r\mathbf{B}}$ is a characteristic function of the set $r\mathbf{B}.$

Let $j$ be the positive integer such that $\frac{1}{\sqrt{j}}\leq r.$

We are going to estimate the norm of the operator $T_{k}^{\alpha}$  on the subspace $L^{p}(r\mathbf{B},d\tau).$
For this purpose we will  use  the classical Shur's test (see \cite{zhu}, Theorem 3.6).

 The test function $h_{j}$ is defined to be
$$h_{j}(x)=\left\{\begin{array}{rl}
              (1-|x|^{2})^{c},&0\leq |x|< \frac{1}{\sqrt{j}}\\
                            0,& \frac{1}{\sqrt{j}}\leq|x|<1
              \end{array}
              \right.$$

  where $c$ is a positive number such that $c\geq \max{\{\frac{n+|k|+\alpha-1}{q},\frac{|k|}{p}\}}.$

We have,
\begin{equation}\label{test12}
\begin{split}
\int_{\mathbf{B}}\frac{(1-|y|^{2})^{n+\alpha-1}h_{j}^{p}(y)}{[x,y]^{n+|k|+\alpha-1}}&d\tau(y)\\
&\leq\int_{\mathbf{B}}\frac{(1-|y|^{2})^{pc+\alpha-1}}{[x,y]^{n+|k|+\alpha-1}}dv(y)\\
&=I_{pc+\alpha,-pc+|k|}(x)\\
&\leq C(pc+\alpha,-pc+|k|)\\
&\leq \left(\frac{j}{j-1}\right)^{pc}C(pc+\alpha,-pc+|k|)h_{j}^{p}(x),
\end{split}
\end{equation}
and
\begin{equation}\label{test14}
\begin{split}
\int_{\mathbf{B}}\frac{(1-|y|^{2})^{n+\alpha-1}h_{j}^{q}(x)}{[x,y]^{n+|k|+\alpha-1}}&d\tau(x)\\
&\leq\int_{\mathbf{B}}\frac{(1-|y|^{2})^{n+\alpha-1}(1-|x|^{2})^{qc-n}}{[x,y]^{n+|k|+\alpha-1}}dv(x)\\
&= (1-|y|^{2})^{n+\alpha-1} I_{qc-n+1,-qc+n+|k|+\alpha-1}(y)\\
&\leq\left(\frac{j}{j-1}\right)^{qc} C(qc-n+1,-qc+n+|k|+\alpha-1)h_{j}^{q}(y).
\end{split}
\end{equation}

 The inequalities \eqref{test12} and \eqref{test14} imply
\begin{equation}
\begin{split}
\label{setovi}
&\|T_{k}^{\alpha}\|_{L^{p}(r\mathbf{B},d\tau)\rightarrow L^{p}(r\mathbf{B},d\tau)} \\
&\leq c_{\alpha}\left(\frac{j}{j-1}\right)^{2c}C^{\frac{1}{p}}(pc+\alpha,-pc+|k|)C^{\frac{1}{q}}(qc-n+1,-qc+n+|k|+\alpha-1)\\
&=c_{\alpha}\left(\frac{j}{j-1}\right)^{c}\pi^{n/2}\Gamma(\alpha)\left(\frac{\Gamma(pc-|k|)}{\Gamma(pc-\frac{|k|-\alpha-1}{2})\Gamma(pc-\frac{|k|-n-\alpha+1}{2})}\right)^{1/p}\\
&\times\left(\frac{\Gamma(qc-n-|k|-\alpha+1)}{\Gamma(qc-n-\frac{|k|+\alpha-3}{2})\Gamma(qc-\frac{|k|+n+\alpha-1}{2})}\right)^{1/q}.\\
\end{split}
\end{equation}
Taking $c=\frac{n+|k|+\alpha-1}{q}+\frac{|k|}{p}$ and
letting $j\rightarrow +\infty$ in \eqref{setovi}
 we obtain
 $$\|T_{k}^{\alpha}\|_{L^{p}(r\mathbf{B},d\tau)\rightarrow L^{p}(r\mathbf{B},d\tau)}\leq \tilde{D}_{p}^{|k|},$$
 where
 \begin{equation}
\label{tojeto}
\begin{split}
&\tilde{D}_{p}^{|k|}=\Gamma(\frac{n}{2}+\alpha)\\
&\times\frac{\Gamma^{1/p}(\frac{p}{q}(n+|k|+\alpha-1))\Gamma^{1/q}(\frac{q}{p}|k|)}{\Gamma^{1/p}(\frac{p}{q}(n+|k|+\alpha-1)+\frac{|k|+\alpha+1}{2})\Gamma^{1/p}(\frac{p}{q}(n+|k|+\alpha-1)+\frac{|k|+n+\alpha-1}{2})}\\
&\times\frac{1}{\Gamma^{1/q}(\frac{q}{p}|k|+\frac{|k|+\alpha+1}{2})\Gamma^{1/q}(\frac{q}{p}|k|+\frac{|k|+n+\alpha-1}{2})}.\\
\end{split}
\end{equation}
Further, since the function $\psi(x)=\log{\Gamma(x)}, x>0$ is convex, the Jensen's inequality give us the new inequality

\begin{equation}
\begin{split}
&\frac{1}{p}\psi\left(\frac{p}{q}(n+|k|+\alpha-1)+\frac{|k|+\alpha+1}{2}\right)+\frac{1}{q}\psi\left(\frac{q}{p}|k|+\frac{|k|+\alpha+1}{2}\right)\\
&+\frac{1}{p}\psi\left(\frac{p}{q}(n+|k|+\alpha-1)+\frac{|k|+n+\alpha-1}{2}\right)+\frac{1}{q}\psi\left(\frac{q}{p}|k|+\frac{|k|+n+\alpha-1}{2}\right)\\
&\geq \psi\left(\frac{n+\alpha-1}{q}+|k|+\frac{|k|+\alpha+1}{2}\right)+\psi\left(\frac{n+\alpha-1}{q}+|k|+\frac{|k|+n+\alpha-1}{2}\right),
\end{split}
\end{equation}
which implies
$$\tilde{D}_{p}^{|k|}\leq \frac{\Gamma(\frac{n}{2}+\alpha)\Gamma^{1/p}(\frac{p}{q}(n+|k|+\alpha-1))\Gamma^{1/q}(\frac{q}{p}|k|)}{\Gamma(\frac{n+\alpha-1}{q}+|k|+\frac{|k|+\alpha+1}{2})\Gamma(\frac{n+\alpha-1}{q}+|k|+\frac{|k|+n+\alpha-1}{2})}.$$

We aim to prove that the obtained upper bound stays preserved on the whole space $L^{p}(\mathbf{B},d\tau).$

For this purpose, let us denote by $p$ the functional defined on $L^{p}(r\mathbf{B},d\tau)$ with
$$p(u)=\|T_{k}^{\alpha}u\|_{L^{p}(r\mathbf{B},d\tau)},\enspace u\in L^{p}(r\mathbf{B},d\tau) .$$
It is clear that $p$ is sublinear functional, where
\begin{equation}
\label{normada}
\|p\|\leq \|T_{\alpha}^{k}\|_{L^{p}(r\mathbf{B},d\tau)\rightarrow L^{p}(r\mathbf{B},d\tau)}.
\end{equation}
The Han-Banach extension theorem implies that the extended functional $\tilde{p}: L^{p}(\mathbf{B},d\tau)\rightarrow \mathbb{R}$ has a bounded norm
$$\|\tilde{p}\|\leq \|T_{\alpha}^{k}\|_{L^{p}(r\mathbf{B},d\tau)\rightarrow L^{p}(r\mathbf{B},d\tau)},$$

i.e.,

\begin{equation*}
\|T_{\alpha}^{k}\|_{L^{p}(\mathbf{B},d\tau)\rightarrow L^{p}(\mathbf{B},d\tau)}\leq \|T_{\alpha}^{k}\|_{L^{p}(r\mathbf{B},d\tau)\rightarrow L^{p}(r\mathbf{B},d\tau)}.
\end{equation*}

\end{proof}

\begin{remark}
Let us point out the fact that the Stirling's asymptotic formula implies
$$\lim_{p\rightarrow +\infty} D_{p}^{|k|}=+\infty.$$
\end{remark}
In the next theorem we give the two-side norm estimate for the operator $T_{k}^{\alpha}.$
\begin{theorem}
\label{opasnat}
 Given multi-index $k,$ $|k|=m,$ there are constants $D_{p}^{m}$ and $A_{p}^{m}$ such that

\begin{equation}
\label{thm1}
A_{p}^{m}< \|T_{\alpha}^{k}\|_{L^{p}(\mathbf{B},d\tau)\rightarrow L^{p}(\mathbf{B},dv_{mp-n})}\leq D_{p}^{m},
\end{equation}

where  $D_{p}^{m}$ is already defined in \eqref{opasno} and
\begin{equation}
A_{p}^{m}=\frac{\pi^{n/2}\Gamma(p(m+\alpha)+\frac{n}{2}+1)\Gamma(m+\alpha+\frac{n}{p}+1)}{\Gamma(p(m+\alpha)+1)\Gamma(m+\alpha+\frac{n}{2}+\frac{n}{p}+1)}\left(\frac{\Gamma(pm-n+1)}{\Gamma(pm-\frac{n}{2}+1)}\right)^{1/p}.
\end{equation}
\end{theorem}

\begin{proof}

At the beginning, we should point an easy inequality

$$\|T_{k}^{\alpha}f\|_{L^{p}(\mathbf{B},dv_{pm-n})}\leq \|T_{k}^{\alpha}f\|_{L^{p}(\mathbf{B},d\tau)},\enspace f\in L^{p}(\mathbf{B},d\tau),$$
i.e.,
\begin{equation}\label{akoako}\|T_{k}^{\alpha}\|_{L^{p}(\mathbf{B},d\tau)\rightarrow L^{p}(\mathbf{B},dv_{pm-n})}\leq \|T_{k}^{\alpha}\|_{L^{p}(\mathbf{B},d\tau)\rightarrow L^{p}(\mathbf{B},d\tau)}.\end{equation}

Further, we will consider the test function
$$\psi_{k}(x)=\beta^{-1}(1-|x|^{2})^{m+\alpha+\frac{n}{p}}, \enspace x\in \mathbf{B},$$
$\beta=\left(\frac{n\pi^{n/2}\Gamma(p(m+\alpha)+1)}{\Gamma(p(m+\alpha)+\frac{n}{2}+1)}\right)^{\frac{1}{p}}.$

Obviously, $\psi_{k}\in L^{p}(\mathbf{B},d\tau)$ and $\|\psi_{k}\|_{L^{p}(\mathbf{B},d\tau)}=1.$

The Lemma \eqref{konstanta}  implies the following sequence of inequalities:
\begin{equation}
\begin{split}
&\|T_{k}^{\alpha}\psi_{k}\|_{L^{p}(\mathbf{B},dv_{pm-n})}^{p}\\
&=\int_{\mathbf{B}}(1-|x|^{2})^{pm-n}\left|\int_{\mathbf{B}}\frac{\psi_{k}(y)}{[x,y]^{n+m+\alpha-1}}dv(y)\right|^{p}dv(x)\\
&=\beta^{-p}\int_{\mathbf{B}}(1-|x|^{2})^{pm-n}\left|\int_{\mathbf{B}}\frac{(1-|y|^{2})^{m+\alpha+\frac{n}{p}}}{[x,y]^{n+m+\alpha-1}}dv(y) \right|^{p}dv(x)\\
&= \beta^{-p}\int_{\mathbf{B}}(1-|x|^{2})^{pm-n}I_{(m+\alpha+\frac{n}{p}+1,-\frac{n}{p}-1)}^{p}(x)dv(x)\\
&\geq \beta^{-p}\min_{0\leq x\leq 1}I_{(m+\alpha+\frac{n}{p}+1,-\frac{n}{p}-1)}^{p}(x)\int_{\mathbf{B}}(1-|x|^{2})^{pm-n}dv(x)\\
&\geq \beta^{-p}\min_{0\leq x\leq 1}I_{(m+\alpha+\frac{n}{p}+1,-\frac{n}{p}-1)}^{p}(x)\frac{n\pi^{n/2}\Gamma(pm-n+1)}{\Gamma(pm-\frac{n}{2}+1)}.
\end{split}
\end{equation}
Finally,
\begin{equation}
\begin{split}
\|T_{k}^{\alpha}\|_{L^{p}(\mathbf{B},d\tau)\rightarrow L^{p}(\mathbf{B},dv_{pm-n})}&> \frac{\pi^{n/2}\Gamma(p(m+\alpha)+\frac{n}{2}+1)\Gamma(m+\alpha+\frac{n}{p}+1)}{\Gamma(p(m+\alpha)+1)\Gamma(m+\alpha+\frac{n}{2}+\frac{n}{p}+1)}\\
&\times\left(\frac{\Gamma(pm-n+1)}{\Gamma(pm-\frac{n}{2}+1)}\right)^{1/p}.
\end{split}
\end{equation}

\end{proof}
Now we are ready to prove the main result of this section.

\begin{theorem} Given $\alpha>0,$ there are constants $B_{p}^{m}$ and $D_{p}^{m}$ such that
\begin{equation}
\label{glavn}
B_{p}^{m}< \|P_{\alpha}\|_{L^{p}(\mathbf{B},d\tau)\rightarrow B^{p}}\leq C_{\alpha}^{m} {m+n-1\choose m}D_{p}^{m},
 \end{equation}

 \begin{equation}
 \begin{split}
 B_{p}^{m}&=\frac{\Gamma(
 \frac{n}{2})\Gamma(m+1)\Gamma(m+\frac{n}{2}+\alpha)\Gamma(\frac{n}{p}+\alpha)}{2\pi^{n/2}\Gamma(m+\frac{n}{p}+\alpha+\frac{n}{2})}\\
 &\times \left(\frac{n(n+m-1)\Gamma(\frac{n}{2})\Gamma(pm-n+1)}{2\Gamma(pm-\frac{n}{2}+1)}\right)^{1/p}.
 \end{split}
 \end{equation}
\end{theorem}
\begin{proof}
{\bf The upper estimate}

First of all, we shall underline the initial inequality
$$|\partial_{k}^{m}P_{\alpha}f|\leq C_{\alpha}^{m}\sum_{|k|=m}|T_{\alpha}^{k}|,$$
for any multi-index $k,$ such that $|k|=m.$

Jensen's inequality give us
\begin{equation}
\begin{split}
\|P_{\alpha}f\|_{B^{p}}^{p}&=\int_{\mathbf{B}}(1-|x|^{2})^{mp-n}|\partial^{m} P_{\alpha}f(x)|^{p}dv(x)\\
&\leq {m+n-1\choose m}^{p-1}\sum_{|k|=m}\int_{\mathbf{B}}(1-|x|^{2})^{mp-n}|\partial_{k}^{m} P_{\alpha}f(x)|^{p}dv(x)\\
&\leq (C_{\alpha}^{m})^{p} {m+n-1\choose m}^{p-1}\sum_{|k|=m}\int_{\mathbf{B}}(1-|x|^{2})^{mp-n}| T_{\alpha}^{k}f(x)|^{p}dv(x),
\end{split}
\end{equation}
i.e.,
$$\|P_{\alpha}f\|_{B^{p}}\leq C_{\alpha}^{m}{m+n-1\choose m}\|T_{\alpha}^{k}f\|_{ L^{p}(\mathbf{B},dv_{mp-n})}.$$

Finally, the Theorem \eqref{opasnat} gives
$$ \|P_{\alpha}\|_{L^{p}(\mathbf{B},d\tau)\rightarrow B^{p}}\leq C_{\alpha}^{m} {m+n-1\choose m}D_{p}^{m}.$$

{\bf The lower estimate}
We  observe the function  $f_{m}(x)=c^{-1}Z_{m}(x,x_{0})(1-|x|^{2})^{\frac{n}{p}},$ where $m\in \mathbb{N},$ we repeat, is a fixed non-negative integer which is determined as a degree of derivative in definition of Besov norm in \eqref{normaprava1} and $x_{0}=e_1.$

Since,
\begin{equation*}
\begin{split}
\int_{\mathbf{B}}|Z_{m}(x,x_{0})|^{p}dv(x)&=n\int_{0}^{1}r^{m+n-1}dr\int_{\mathbb{S}}|Z_{m}(\xi,x_{0})|^{p}d\sigma(\xi)\\
&\leq \frac{n|\mathbb{S}|}{n+m-1}\left(\dim{H_{m}(\mathbb{R}^{n})}\right)^{p},\\
\end{split}
\end{equation*}
we chose constant $c$ to be $c= \left(\frac{n|\mathbb{S}|}{n+m-1}\right)^{1/p}\dim{H_{m}(\mathbb{R}^{n})},$
 where $\dim{H_{m}(\mathbb{R}^{n})}={n+m-1\choose n-1}-{n+m-3\choose n-1}$ (see \eqref{dimension}. Therefore, $\|f_{m}(x)\|_{L^{p}(\mathbf{B},d\tau)}\leq 1.$

On the other hand,
\begin{equation}
\begin{split}
P_{\alpha}f_{s}(x)&=\int_{\mathbf{B}}R_{\alpha}(x,y)f_{s}(y)dv_{\alpha}(y)\\
&=c^{-1}c_{\alpha}\omega_{\alpha}\sum_{d=0}^{\infty}\int_{\mathbf{B}}A_{d}^{n,\alpha}Z_{d}(x,y)Z_{m}(y,x_0)(1-|y|^{2})^{\frac{n}{p}+\alpha-1}dv(y)\\
&=c^{-1}c_{\alpha}\omega_{\alpha}A_{m}^{n,\alpha}\int_{0}^{1}(1-t^{2})^{\frac{n}{p}+\alpha-1}t^{2m+n-1}dt\int_{\mathbb{S}}Z_{m}(x,\xi)Z_{m}(\xi,x_0)d\sigma(\xi)\\
&=c^{-1}\frac{\Gamma(\frac{n}{2})\Gamma(m+\frac{n}{2}+\alpha)\Gamma(\frac{n}{p}+\alpha)}{2\pi^{n/2}\Gamma(\alpha)\Gamma(m+\frac{n}{p}+\alpha+\frac{n}{2})}Z_{m}(x,x_0)\\
&=M_{\alpha,p}(m,x_0)Z_{m}(x,x_0),
\end{split}
\end{equation}
where
$M_{\alpha,p}(m,x_0)=c^{-1}\frac{\Gamma(\frac{n}{2})\Gamma(m+\frac{n}{2}+\alpha)\Gamma(\frac{n}{p}+\alpha)}{2\pi^{n/2}\Gamma(\alpha)\Gamma(m+\frac{n}{p}+\alpha+\frac{n}{2})}.$
\begin{equation*}
\begin{split}
&\|P_{\alpha}f\|_{L^{p}(\mathbf{B},dv_{mp-n})}^{p}\\
&=\left(M_{\alpha,p}(m,x_0)\right)^{p}\int_{\mathbf{B}}(1-|x|^{2})^{pm-n}\left|\partial^{m}P_{\alpha}f(x)\right|^{p}dv(x)\\
&=\left(M_{\alpha,p}(m,x_0)\right)^{p}\int_{\mathbf{B}}(1-|x|^{2})^{pm-n}\left|\sum_{|k|=m}\partial_{k}^{m}Z_{m}(x,x_0)\right|^{p}dv(x)\\
&= \left(m!\dim{H_{m}(\mathbb{R}^{n})}M_{\alpha,p}(m,x_0)\right)^{p}\int_{\mathbf{B}}(1-|x|^{2})^{pm-n}dv(x)\\
&= \left(m!\dim{H_{m}(\mathbb{R}^{n})}M_{\alpha,p}(m,x_0)\right)^{p} \frac{n\pi^{n/2}\Gamma(pm-n+1)}{\Gamma(pm-\frac{n}{2}+1)}.
\end{split}
\end{equation*}
\end{proof}

\begin{remark}
According to the introduced semi-norm \eqref{normaprava1} the (semi)inner-product associate with the space $B^{2}$ (harmonic Dirichlet space) is
\begin{equation}
\label{scalarp}
\left<f,g\right>=\sum_{|k|=m}\int_{\mathbf{B}}(1-|x|^{2})^{2m}\frac{\partial^{m}f}{\partial x^{k}}(x)\overline{ \frac{\partial^{m}g}{\partial x^{k}}(x)}d\tau(x),\enspace f,g\in B^{2}.
\end{equation}
Finding the "exact" operator norm for the harmonic Bergman projection in the Hilbert case, $P_{\alpha}:L^{2}(\mathbf{B},d\tau)\rightarrow B^{2},$ is another related problem which stays open.
\end{remark}
\begin{conjecture}
For $1<p<\infty,$
$$\|P_{\alpha}\|_{L^{p}(\mathbf{B},d\tau)\rightarrow B^{p}}\approx {m+n-1\choose m}D_{p}^{m}.$$
\end{conjecture}
In the correspondence to the main result from \cite{Boo1}, the related question to the above conjecture would be the examining of the behavior for the norm $\|P_{\alpha}\|_{L^{p}(\mathbf{B},d\tau)\rightarrow B^{p}}$ when $p$ is fixed and $\alpha$ changes with special treatment of the case $\alpha\rightarrow 0.$

\end{document}